\documentclass[11pt, a4paper]{article}
\usepackage{amsmath,amssymb,amsfonts,graphicx,amsthm}

\usepackage{amsthm}
\usepackage[utf8]{inputenc}
\usepackage{graphicx}
\usepackage{authblk}
\usepackage[T1]{fontenc}

\newtheorem{thm}{{\bf  Theorem}}
\newtheorem{cor}{{\bf  Corollary}}
\newtheorem{pro}{{\bf  Proposition}}
\newtheorem{lem}{{\bf  Lemma}}
\input epsf
\begin{document}

\title{\textbf{The Metric Dimension of The Tensor Product of Cliques}}

\author{{\normalsize
{\sc H. Amraei ${}^{\mathsf{a}}$},\,
  {\sc H.R. Maimani ${}^{\mathsf{a}, \mathsf{b}}$},\,
  {\sc A. Seify ${}^{\mathsf{a}, \mathsf{b}}$}\,
  {\sc and A. Zaeembashi ${}^{\mathsf{a}}$\,}
 \vspace{3mm}
\\{\footnotesize{${}^{\mathsf{a}}$\it Mathematics Section, Department of Basic Sciences, Shahid Rajaee Teacher Training University,}}{\footnotesize{}}\\{\footnotesize{${}^{\mathsf{}}$\it
P.O. Box 16783-163,  Tehran, Iran.}}
{\footnotesize{}}\\{\footnotesize{${}^{\mathsf{b}}$\it School of Mathematics, Institute for Research in Fundamental Sciences (IPM),}}{\footnotesize{}}\\{\footnotesize{${}^{\mathsf{}}$\it
P.O. Box 19395-5746,
 Tehran, Iran.}}
\thanks{{\it E-mail addresses}:
$\mathsf{amrai.hadi@yahoo.com}$, $\mathsf{maimani@ipm.ir}$, $\mathsf{abbas.seify@gmail.com}$ and $\mathsf{azaeembashi@srttu.edu}$.} } }

\date{}

\maketitle

\begin{abstract}
Let $G$ be a connected graph and $W = \{w_1, w_2, \ldots, w_k\} \subseteq V(G)$ be an ordered set. For every vertex $v$,
the metric representation of $v$ with respect to $W$ is an ordered $k$-vector defined as $r(v|W) := (d(v, w_1), d(v, w_2), \ldots, d(v, w_k))$, where $d(x, y)$ is the distance between the vertices $x$ and $y$. The set $W$ is called a resolving set for $G$ if distinct
vertices of $G$ have distinct representations with respect to $W$.
The minimum cardinality of a resolving set for $G$ is its metric dimension and is denoted by $dim(G)$. 
In this paper, we study the metric dimension of tensor product of cliques and prove some bounds. Then we determine the metric dimension of tensor product of two cliques. 
\end{abstract}



\section{Introduction}
Let $G=(V,E)$ be a finite simple graph. The {\it distance} between two vertices $u$ and $v$, denoted by $d(u, v)$, is the length of a shortest path between $u$ and $v$ in $G$. The diameter of $G$ is the maximum distance between the vertices of $G$ and is denoted by $diam(G)$. Suppose that $G$ and $H$ are two simple graphs. The {\it tensor product} of $G$ and $H$ is denoted by $G \otimes H$ and is a graph with $V(G \otimes H)=\{(u,v): u\in V(G), v\in V(H)\}$ and two vertices $(u,v)$ and $(x,y)$ are adjacent if and only if $ux \in E(G)$ and $vy \in E(H)$. For $t\geq 3$, the tensor product of $G_1, \ldots, G_t$ is defined by induction. 
\\
A {\it complete graph} of order $n$ is denoted by $K_n$ and is called a {\it clique}. Also, $K_{m,n}$ denotes the {\it complete bipartite graph}, whose two parts are of size $m$ and $n$. A subset $M \subseteq E(G)$ is called a {\it matching}, if no two edges in $M$ have a common end vertex. A matching $M$ is called a {\it perfect matching} if every vertex of $G$ is incident with some edge in $M$. Throughout the paper, we suppose that $V(K_m)=\{u_1, \ldots, u_m\}$ and $V(K_n)=\{v_1, \ldots, v_n\}$.
\\
Let $A_1, \ldots, A_k$ be nonempty sets and $T\subseteq A_1 \times \cdots \times A_k$, in which the product is {\it Cartesian product}. By $T(i) \subseteq A_i$ we mean all elements of $A_i$ which are appeared as the $i$-th coordinate of some element of $T$.
\\
For an ordered set $W = \{w_1, w_2, \ldots, w_k\}$ of vertices and a vertex $v$ in a connected
graph $G$, the ordered $k$-vector $r(v \, | \, W) := (d(v, w_1), d(v, w_2), \ldots, d(v, w_k))$ is called
the {\it metric representation} of $v$ with respect to $W$. The set $W$ is called a {\it resolving set} of $G$, if distinct
vertices of $G$ have distinct representations with respect to $W$. The minimum cardinality of a resolving set for $G$ is its {\it metric dimension} and is denoted by $dim(G)$. The metric dimension in general graphs was firstly studied by Harary and
Melter \cite{harary}, and independently by Slater \cite{slater}. In graph theory, metric dimension is a parameter that has appeared in various applications, as diverse as network discovery and verification \cite{app1}, strategies for the Mastermind game \cite{app2}, combinatorial optimization \cite{app3} and so on. 
\\
Finding the parameters of products of graphs is one of the well-known problems in graph theory. The metric dimension of the Cartesian product of graphs is studied in \cite{caceres}. Also, Jannesari and Omoomi studied the metric dimension of the Lexicographic product of graphs, see \cite{omoomi1}. In this article, we study the metric dimension of the tensor product of cliques. Also, we determine the metric dimension of the tensor product of two cliques. Our main result is as follows.

\begin{thm}\label{main}
Let $G=K_m \otimes K_n$ and $n \geq m$. Then
\begin{enumerate}
\item If $m=n=2$, then $G$ is disconnected.
\item If $n \leq 2m-2$ and $m \geq 3$. Then $dim(G) = \lceil \frac{2}{3}(m+n-2) \rceil$.
\item If $n\geq 2m-1$ and $m\geq 2$. Then $dim(G)= n-1$. 
\end{enumerate}
\end{thm}

In section 2, we prove some bounds for the metric dimension of tensor product of cliques and section 3 is dedicated to the proof of Theorem \ref{main}.

\section{Some Bounds }
In this section, we present some bounds for the metric dimension of $K_{m_1} \otimes \cdots \otimes K_{m_t}$. First, we present two following lemmas. The proof of the first lemma is easy and we omit it.

\begin{lem}\label{lem1}
Let $G=K_{m_1} \otimes \cdots \otimes K_{m_t}$ and $m_i \leq m_{i+1}$. Then the followings hold:
\begin{enumerate}
\item If $m_1=m_2=2$, then $G$ is disconnected.
\item If $m_1=2$ and $m_2 \geq 3$, then $diam(G)=3$.
\item If $m_1\geq 3$, then $diam(G)=2$.
\end{enumerate}
\end{lem}

\begin{lem} \label{lem}
Let $G=K_{m _1} \otimes \cdots \otimes K_{m_t}$ with $m_i \geq 3$ and $W$ be a resolving set of $G$. Then $|V(K_{m_i}) \setminus W(i)| \leq 1$.
\end{lem}

\begin{proof}
By Lemma \ref{lem1}, we conclude that $diam(G)=2$. Note that if $u=(u_1, \ldots, u_t)$ and $v=(v_1, \ldots, v_t)$, then $d(u,v)=2$ if and only if $u_i=v_i$, for some $1 \leq i \leq t$.
\\
On the contrary, suppose that there exist $1 \leq i \leq t$ and two distinct vertices $x, y \in K_{m_i}$, such that $x, y \notin W(i)$. With no loss of generality, we may assume that $i=1$.
Let $u_i \in V(K_{m_i})$, for $i=2, \ldots, t $. It is not hard to see that $r((x, u_2, \ldots, u_t) \, | \, W)=r((y, u_2, \ldots, u_t) \, | \, W)$, which is a contradiction. Hence $|V(K_{m_i}) \setminus W(i)| \leq 1$.
\end{proof}

Now, the following corollary is clear.

\begin{cor}\label{cor}
Let $G=K_{m_1} \otimes \cdots \otimes K_{m_t}$, where $m_i \geq 3$. Then $dim(G) \geq max\, \{m_i-1: i=1, \ldots, t \}$.
\end{cor}

In the following propositions, we present upper and lower bounds for $dim(K_{m_1} \otimes \cdots \otimes K_{m_t})$.

\begin{pro}
Let $G=K_{m_1} \otimes \cdots \otimes K_{m_t}$, where $m_i \geq 3$. Then
\begin{center}
$dim(G) \geq max \, \{dim(K_{i_{1}} \otimes \cdots \otimes K_{i_{t-1}})\}, $
\end{center}
where $\{i_1, \ldots, i_{t-1}\} \subseteq \{m_1, \ldots, m_t\}$.
\end{pro}

\begin{proof}
Let $W$ be a resolving set of $G$. Define $W(\overline{1})$ as follows.
\vspace{0.5em}
\begin{center}
$W(\overline{1})=\{ (z_2, \ldots, z_t): \;\;\; (z_1, z_2, \ldots, z_t) \in W,\;\; for \; some \; z_1 \in V(K_{m_1}) \}$.
\end{center}
\vspace{0.5em}
If $|W(\overline{1})|<dim(K_{m_{2}} \otimes \cdots \otimes K_{m_{t}})$, then there exist $x,y \in V(K_{m_2} \otimes \cdots \otimes K_{m_t})$ such that $r(x \, | \, W(\overline{1}))=r(y \, | \, W(\overline{1}))$. Suppose that $x=(x_2, \ldots, x_t)$ and $y=(y_2, \ldots, y_t)$. It is clear that $r((v,x_2, \ldots, x_t) \, | \, W)=r((v,y_2, \ldots, y_t) \, | \, W)$, for every $v \in V(K_{m_1})$. This is a contradiction and implies that $|W(\overline{1})| \geq dim(K_{m_{2}} \otimes \cdots \otimes K_{m_{t}})$. One can prove a similar relation for $|W(\overline{i})|$, for $i=1, 2, \ldots, t$. Clearly, $|W(\overline{i})| \leq |W|$ and this completes the proof. 
\end{proof}

Note that Proposition \ref{case2} will show that the previous bound is sharp. Now, we present an upper bound for the metric dimension of the tensor product of cliques.

\begin{pro}
Let $G=K_{m_1} \otimes \cdots \otimes K_{m_t}$, where $m_i \geq 3$. Then
\begin{center}
$dim(G) \leq 3\;min \, \{dim(K_{i_{1}} \otimes \cdots \otimes K_{i_{t-1}})+ dim(K_{j_{1}} \otimes \cdots \otimes K_{j_{t-1}}) \},$
\end{center}
where $\{i_1, \ldots, i_{t-1}\}$ and $\{j_1, \ldots, j_{t-1}\}$ are two distinct subsets of $\{m_1, \ldots, m_t\}$.
\end{pro}

\begin{proof}
Let $a_1, a_2, a_3 \in V(K_{m_1})$ and $b_1, b_2, b_3 \in V(K_{m_t})$. Also, suppose that $W_1$ and $W_t$ are resolving sets for $K_{m_2} \otimes \cdots \otimes K_{m_t}$ and $K_{m_1} \otimes \cdots \otimes K_{m_{t-1}}$, respectively. Define $W'_1$ and $W'_t$ as follows.
\vspace{0.5em}
\begin{center}
$W'_1=\{ (a_i, u_2, \ldots, u_t): \; i=1, 2, 3 \;\;\; (u_2, \ldots, u_t) \in W_1 \}$
\\
$W'_t=\{ (u_1, \ldots, u_{t-1}, b_i): \; i=1, 2, 3 \;\;\; (u_1, \ldots, u_{t-1}) \in W_t \}$.
\end{center}
\vspace{0.5em}
Now, we show that $W=W'_1 \cup W'_t$ is a resolving set for $G$. Let $x=(x_1, \ldots, x_t)$ and $y=(v_1, \ldots, y_t)$ be two distinct vertices of $G$. Thus $x_i \neq y_i$, for some $1 \leq i \leq t$. With no loss of generality, we assume that $i=1$ and $b_3 \neq x_t, y_t$. Let $X=\{ (u_1, \ldots, u_{t-1}, b_3): (u_1, \ldots, u_{t-1}) \in W_t \}$, then $r(x\,|\,X) \neq r(y\,|\,X)$ and therefore $r(x\,|\,W) \neq r(y\,|\,W)$. Note that if $i \neq 1$, then the similar method, by using $W'_{1}$, implies that $r(x\,|\,W) \neq r(y\,|\,W)$. This implies that $W$ is a resolving set of $G$ and completes the proof.
\end{proof}

\section{Proof of Theorem \ref{main}}

This section is dedicated to the proof of Theorem \ref{main}. We divide the proof into three parts, i.e. Propositions \ref{case1}, \ref{case2} and \ref{case3}. In \cite{bipartite} the following theorem was proved.

\begin{thm}\label{bi}
Let $G=K_{n,n} \setminus I$, in which $I$ is a perfect matching and $n \geq 3$. Then $dim(G)=n-1$.
\end{thm}

Now, we prove the following proposition.

\begin{pro}\label{case1}
Let $G=K_2 \otimes K_n$. Then
\begin{enumerate}
\item If $n=2$, then $G$ is disconnected.
\item If $n \geq 3$, then $dim(G)= n-1$.
\end{enumerate}
\end{pro}

\begin{proof}
It is clear that $K_2 \otimes K_n$ is isomorphic to $K_{n,n} \setminus I$, where $I$ is a perfect matching. Now, Theorem \ref{bi} completes the proof.
\end{proof}

So, we may assume that $m,n \geq 3$. First, we prove the following theorem.

\begin{pro}\label{case2}
Let $G=K_m \otimes K_n$ such that $n\geq 2m-1$ and $m\geq 3$. Then $dim(G)= n-1$. 
\end{pro}

\begin{proof}
Note that Corollary \ref{cor} implies $dim(G) \geq n-1$. We construct a resolving set of size $n-1$ and this completes the proof. Define $W\subseteq V(G)$ as follows:
\vspace{0.5em}
\begin{center}
$W=\{(u_1, v_1), \ldots, (u_{m-1}, v_{m-1}), (u_1, v_m), \ldots, (u_{m-1}, v_{2m-2})\} \cup \{(u_1,v_j): 2m-1 \leq j \leq n-1\}$
\end{center}
\vspace{0.5em} 
Clearly, $|W|=n-1$. Now, we show that $W$ is a resolving set of $G$. Let $(u,v)$ and $(x,y)$ be two distinct vertices of $G$. Without loss of generality, we may assume that $(u, v)$ and $(x, y)$ are not in $W$. Now, we consider two cases.
\vspace{1.5em}
\\
\textbf{Case 1.} $u=x$ or $v=y$.
\\
Suppose that $u=x$. Thus $v\neq y$ and hence $v$ or $y$ is not equal to $v_n$. Suppose that $y\neq v_n$. Therefore, there exists an element $u_k\in V(K_m)$, such that $(u_k, y)\in W$. Since $(u, y)\notin W$, we conclude that $u_k \neq u$. 
Then the coordinates corresponding to $(u_k, y)$ in $r((u,v)\,|\,W)$ and $r((u,y)\,|\,W)$ have different values. 
\vspace{1.5em}
\\
\textbf{Case 2.} $u \neq x$ and $v \neq y$. We consider two subcases:
\vspace{0.8em}
\\
\textbf{Subcase 2.1.} $u_m \in \{u, x\}$ or $v_n \in \{v, y\}$.
\\
Let $u= u_m$. If $v=v_n$, then $r((u,v)|W)=(1, \ldots, 1)$ and
this implies that $r((u,v)\,|\,W) \neq r((x,y)\,|\,W)$. So, we may assume that $v \neq v_n$.
\\
Since $u \neq x$, we have $x \in W(1)$ and by the structure of $W$, there are at least two elements in $W$ such as $(x,v_k)$ and $(x,v_l)$. Now, at most one of the coordinates corresponding to $(x,v_k)$ and $(x,v_l)$ in $r((u,v)\,|\,W)$ is equal to two while both of them in $r((x,y)\,|\,W)$ are equal to two. This implies that $r((u,v)\,|\,W) \neq r((x,y)\,|\,W)$. 
\vspace{0.8em}
\\
\textbf{Subcase 2.2.} $u_m \notin \{u, x\}$ and $v_n \notin \{v, y\}$.
\\
By the structure of $W$, there are $(x,v_k)$ and $(x,v_l)$ in $W$ such that $v_k, v_l \neq y$. Now, coordinates corresponding to $(x,v_k)$ and $(x,v_l)$ in $r((x,y)\,|\,W)$ are equal to two. Since $u\neq x$ we can conclude that at most one of these components is equal to two in $r((u,v)\,|\,W)$. This implies that $r((u,v)\,|\,W) \neq r((x,y)\,|\,W)$ and completes the proof.
\end{proof}

We have the following lemma.

\begin{lem}\label{2}
Let $G =K_m \otimes K_n$, where $m,n \geq 3$, $W \subseteq V(G)$ and $(u,v), (x,y) \in W$. Also suppose that for every $(u_1,v_1) \in W\setminus \{(u,v), (x,y)\}$ we have $u_1 \notin \{u,x\}$ and $v_1 \notin \{v,y\}$. Then $r((u,y)\,|\,W)=r((x,v)\,|\,W)$.
\end{lem}

\begin{proof}
It is easy to see that all coordinates of $r((u,y)\,|\,W)$ and $r((x,v)\,|\,W)$ are equal to one, except the coordinates corresponding to $(u,v)$ and $(x,y)$. 
\end{proof}

Now, we determine the metric dimension of $G=K_m \otimes K_n$, where $m \leq n \leq 2m-2$. First, we present the lower bound for $dim(G)$ in the following lemma and in the next proposition we show that this bound is the exact value of $dim(G)$.

\begin{lem}\label{lower}
Let $G=K_m \otimes K_n$, where $m \leq n \leq 2m-2$ and $m \geq 3$. Then $dim(G) \geq \lceil \frac{2}{3} (m+n-2) \rceil$.
\end{lem}

\begin{proof}
Let $W$ be a minimal resolving set and $M \subseteq W$ be of maximum size, such that if $(u,v), (x,y) \in M$, then $u\neq x$ and $v\neq y$. Suppose that $|M|=k$. Without loss of generality, we may assume that $M=\{(u_{1}, v_{1}), \ldots, (u_{k}, v_{k}) \}$. We find a lower bound for $|W|$. By Lemma \ref{lem}, we have $|V(K_m) \setminus W(1)|, |V(K_n) \setminus W(2)| \leq 1$. Assume that $|V(K_m) \setminus W(1)|=s$ and $|V(K_n) \setminus W(2)|=t$, where $s,t \in \{0, 1\}$.
\\
Let $k<i<m-s$. Since $u_i\in W(1)$, there exists $1\leq j_i \leq n-t$ such that $(u_i, v_{j_i})\in W$. Maximality of $M$ implies $j_i \leq k $.
So, we can find $j_1, \ldots, j_{m-k-s} \leq k$ such that $A=\{(u_{k+1},v_{j_1}), \ldots, (u_{m-s},v_{j_{m-k-s}})\} \subseteq W$. Similarly, there exists $i_1, \ldots, i_{n-k-t} \leq k$ such that $B=\{(u_{i_1},v_{k+1}), \ldots, (u_{i_{n-k-t}},v_{n-t})\} \subseteq W$.
\\
If $(u,v),(x,y) \in M$ such that $u,x \notin A(1) \cup B(1)$ and $v,y \notin A(2) \cup B(2)$, then Lemma \ref{2} implies that $r((u,y)\,|\,M \cup A \cup B)=r((x,v)\,|\,M \cup A \cup B)$. Since $W$ is a resolving set, there exists $(a,b)\in W$ such that $d((u,y),(a,b)) \neq d((x,v),(a,b))$. Now, we consider two cases:
\vspace{0.8em}
\\
\textbf{Case 1.} $s=t=1$.
\\
Note that if $u_m \notin W(1)$ and $v_n \notin W(2)$, then $r((u,v_n)\,|\,M \cup A \cup B)=r((u_m, v)\,|\,M \cup A \cup B)$. This implies that:
\begin{center}
$|W| \geq k+(m-k-1)+(n-k-1)+max\,\{0,\lceil \frac{3k-m-n+2}{2} \rceil \}=f(k)$
\end{center}
Note that if $k \leq \lfloor \frac{m+n-2}{3} \rfloor$, then $f(k)= m+n-k-2$. This yields that $f(k)$ has its minimum value at $k = \lfloor \frac{m+n-2}{3} \rfloor$.
\\
If $k \geq \lceil \frac{m+n-2}{3} \rceil$, then it is not hard to see that $f(k+1) \geq f(k)$. This implies that $f(k)$ has its minimum value at $k = \lceil \frac{m+n-2}{3} \rceil$.
\\
Note that $f(k)=f(k')$, where $k = \lfloor \frac{m+n-2}{3} \rfloor$ and $k' = \lceil \frac{m+n-2}{3} \rceil$. So, we have:
\begin{center}
$|W| \geq k+(m-k-1)+(n-k-1)$,
\end{center}
where $k= \lfloor \frac{m+n-2}{3} \rfloor$. Hence, if $W$ is an arbitary minimal resolving set of $G$ and $s=t=1$, then $|W| \geq \lceil \frac{2}{3}(m+n-2) \rceil$.
\vspace{0.8em}
\\
\textbf{Case 2.} $s=0$ or $t=0$.
\\
Assume that $s=0$. We have:
\begin{center}
$ |W| \geq k+(m-k)+(n-k-t)+max\,\{0,\lfloor \frac{3k-m-n+t}{2} \rfloor \}=f(k) $
\end{center}
Then the similar method, to the previous case, shows that $f(k)$ has its minimum value in $k= \lfloor \frac{m+n-t}{3} \rfloor$ and therefore $|W| \geq \lceil \frac{2}{3}(m+n-t) \rceil$.
\vspace{0.8em}
\\
Hence if $W$ is an arbitrary minimal resolving set of $G$, then $|W| \geq \lceil \frac{2}{3}(m+n-2) \rceil$.
\end{proof}

Now, we prove the following proposition which completes the proof of Theorem \ref{main}.

\begin{pro}\label{case3}
Let $G=K_m \otimes K_n$ such that $m \leq n \leq 2m-2$ and $m \geq 3$. Then $dim(G) = \lceil \frac{2}{3}(m+n-2) \rceil$.
\end{pro}

\begin{proof}
By Lemma \ref{lower}, it sufficies to find a resolvable set of size $\lceil \frac{2}{3}(m+n-2) \rceil$.  
Let $V$ be defined as follows and $k=\lfloor \frac{m+n-2}{3} \rfloor$.
\begin{center}
$V_1=\{\,(u_1,v_1), \ldots, (u_k, v_k)\}$\\
$V_2= \{(u_{k+i}, v_{i})\,\, | \,\, 1 \leq i \leq m-k-1\,\}$\\
$V_3= \{\,(u_{m-k-1+i}, v_{k+i})\,\, |\,\, 1\leq i \leq n-k-1\, \}$\\
$V= V_1 \cup V_2 \cup V_3$, 
\end{center}
where we use mod $k$ arithmetic for indices of $v_i$ and $u_{m-k-1+i}$ in $V_2$ and $V_3$, respectively. We claim that $W$ is a resolving set of $G$. Note that $|V_2 \cup V_3| \geq |V_1|$ and this implies that if $(u,v) \in W$, then either there exists $(u,v') \in W$ or there exists $(u',v) \in W$, where $u' \neq u$ and $v' \neq v$.
\\
Let $(u,v)$ and $(x,y)$ be two distinct vertices of $G$. Without loss of generality, we assume that $(u, v)$ and $(x, y)$ are not in $W$. We consider two cases.
\vspace{1.5em}
\\
\textbf{Case 1.} $u=x$ or $v=y$.\\
Suppose that $u=x$. Therefore, $v \neq y$ and we may assume that $y \neq v_m$. Hence $y \in W(2)$. 
So, there exists $(u_k,y) \in W$, where $u_k \neq u$. Then the coordinate corresponding to $(u_k, y)$ in $r((u,v)\,|\,W)$ is equal to one and in $r((u,y)\,|\,W)$ is two.
\vspace{1.5em}
\\
\textbf{Case 2.} $u \neq x$ and $v \neq y$. We consider two subcases in this case:
\vspace{0.8em}
\\
\textbf{Subcase 2.1.} $u_m \in \{u, x\}$ or $v_n \in \{v, y\}$.
\\
Suppose that $u= u_m$. Since $u \neq x$ we have $x \in W(1)$. If $v=v_n$, then $r((u,v)\,|\,W)=(1, \ldots, 1)$ and this implies that $r((u,v)\,|\,W) \neq r((x,y)\,|\,W)$. Also, if $y \in W(2)$, then there exists $(u_k,y) \in W$ such that $u_k \neq x$. Thus, the coordinate corresponding to $(u_k,y)$ in $r((u,v)\,|\,W)$ is one and in $r((x,y)\,|\,W)$ is two. So, we may assume that $y \notin W(2)$ and $v \in W(2)$. 
\\
By contrary, if $r((u,v)\,|\,W)=r((x,y)\,|\,W)$, then both of them have exactly one coordinate equal to two in their representation and this coordinate is corresponding to $(x,v)$. Thus $(x,v) \in W$ and both of $x$ and $v$ are only appeared in $(x,v)$. By the structure of $W$, this is impossible and this completes the proof in this subcase.
\vspace{0.8em}
\\
\textbf{Subcase 2.2.} $u_m \notin \{u, x\}$ and $v_n \notin \{v, y\}$.
\\
By contrary, suppose that $r((u,v)\,|\,W)=r((x,y)\,|\,W)$. Since $u \in W(1)$ and $v \in W(2)$, there exist $(u,v_k),(u_l,v) \in W$ such that $v_k \neq v$ and $u_l \neq u$.\\
If there exists $(u,v_r) \in W$ such that $v_r \neq v_k$, then without loss of generality, we may assume that $v_r \neq y$. This implies that $r((u,v)\,|\,W) \neq r((x,y)\,|\,W)$, because the coordinates corresponding to $(u,v_r)$ in $r((u,v)\,|\,W)$ and $r((x,y)\,|\,W)$ have different values.\\
On the other hand, since $r((u,v)\,|\,W)=r((x,y)\,|\,W)$, so they have the same value at the coordinate corresponding to $(u,v_k)$. This implies that $v_k=y$ and therefore $(u,y) \in W$. Note that if there exists $(s,y) \in W$ such that $s \neq u$, then one can show that $r((u,v)\,|\,W) \neq r((x,y)\,|\,W)$, a contradiction. So, $(u,y) \in W$ with this property that $u$ and $y$ are only appeared in this element of $W$. By the structure of $W$, this is impossible and this completes the proof.
\end{proof}

\end{document}